\documentclass[oneside,english]{amsart}
\usepackage[T1]{fontenc}
\usepackage[latin9]{inputenc}
\usepackage{amsthm}
\usepackage{amstext}
\usepackage{amssymb}
\usepackage{esint}
\usepackage{graphicx}
\usepackage{enumerate}
\usepackage{amscd}
\usepackage{color}

\makeatletter
\newtheorem{theorem}{Theorem}[section]
\newtheorem{lemma}[theorem]{Lemma}

\newtheorem{corollary}[theorem]{Corollary}
\numberwithin{figure}{section}
\theoremstyle{definition}
\newtheorem{definition}[theorem]{Definition}

\theoremstyle{remark}

\numberwithin{equation}{section}
\makeatother

\usepackage{color}

\newcommand{\Om} {\Omega}

\usepackage{babel}

\newcommand\R{\mathbb R}

\begin{document}

\title[On the weighted Steklov eigenvalue problems]
{On the weighted Steklov eigenvalue problems in outward cuspidal domains}

\author{Prashanta Garain, Vladimir Gol'dshtein, Alexander Ukhlov}

\begin{abstract}
In this article, we investigate the weighted Steklov eigenvalue problem and the weighted Schr\"odinger--Steklov  eigenvalue problem in outward cuspidal domains. We prove the solvability of these spectral problems in both linear and non-linear cases.
\end{abstract}

\maketitle



\section{Introduction}

The Steklov-type eigenvalue problems arise in various fields of continuum mechanics, including fluid mechanics and elasticity (see, for example, \cite{BS, CPV}). These problems have garnered increasing attention in recent years (see, for example, \cite{B21, BR, CGGS, FL, MN10, XW}).

The classical Steklov eigenvalue problem \cite{Steklov} in a bounded domain $\Omega \subset \mathbb{R}^n$, $n \geq 2$, with a piecewise smooth boundary $\partial \Omega$, can be formulated as follows:
\begin{equation}
\label{steq}
\begin{cases}
-\Delta u=0\,\,\text{in}\,\, \Omega,\\
\nabla u\cdot\nu =\lambda w u \,\, \text{on}\,\, \partial\Omega,
\end{cases}
\end{equation}
where $\nu$ is the unit outward normal to $\partial \Omega$ and $w$ is a non-negative bounded weight function. In the case of Lipschitz domains $\Omega \subset \mathbb{R}^n$, $n \geq 2$, the classical Steklov eigenvalue problem has a long history and is sufficiently well studied (see \cite{AM, GP17, KLP} and references therein).

In recent years, there has been increasing attention on the geometric analysis of PDEs in cuspidal domains; see, for example, \cite{GU09,GU17,KLP,KUZ,MP}. Recall the notion of outward cuspidal domains \cite{GV,KUZ}. Let $\gamma:[0,1] \to [0,\infty)$ be a continuous, increasing, and differentiable function such that $\gamma(0)=0$ and $\gamma(1)=1$. In addition, let $\gamma'$ be increasing on $(0, 1)$, with $\lim_{t \to 0^{+}} \gamma'(t) = 0$. The basic example of such functions is $\gamma(t) = t^{\alpha}$, where $1 < \alpha < \infty$.  Denote $x'=(x_1,...,x_{n-1})$. Then an outward cuspidal domain $\Omega_{\gamma}\subset\mathbb{R}^n$, $n\geq 2$, is defined by
\begin{multline}
\label{cusp}
\Omega_{\gamma}=\\
\left\{(x',x_n)\in \mathbb R^{n-1}\times \mathbb R : \sqrt{x_1^2+...+x_{n-1}^2}<\gamma(x_n), 0<x_n\leq 1\right\}\cup B^n\left((0,2), \sqrt{2}\right),
\end{multline}
where $B^n\left((0,2), \sqrt{2}\right)\subset\mathbb{R}^n$ is the open ball of radius $\sqrt{2}$ centered at $(0,2)\in\mathbb{R}^{n-1}\times\mathbb{R}$.

In the case of outward cuspidal domains $\Omega_{\gamma} \subset \mathbb{R}^n$, the Steklov eigenvalue problem generally represents an open and complicated problem (see, for example, \cite{AM20}). In \cite{NT}, it was proved that for the cusp function $\gamma(t) = t^{\alpha}$, the unweighted Steklov problem has a discrete spectrum if $1 < \alpha < 2$. If $\alpha = 2$, the spectrum has a continuous part, and the point of the spectrum $\lambda_0 = 0$ belongs to the continuous spectrum for $\alpha > 2$.

In this article, by using the weighted trace embedding theorems \cite{GV}, we suggest the solution of the classical weighted Steklov eigenvalue problem in outward cuspidal domains $\Omega_{\gamma} \subset \mathbb{R}^n$. By leveraging the compactness of the weighted trace embedding operator \cite{GV},
$$
i: W^{1,2}(\Omega_{\gamma}) \hookrightarrow L^2_w(\partial \Omega_{\gamma}),
$$
we demonstrate that this weighted eigenvalue problem has a discrete spectrum, which can be expressed as a non-decreasing sequence:
\[
0 = \lambda_0 < \lambda_1 \leq \lambda_2 \leq \ldots \leq \lambda_k \leq \ldots,
\]
where the weight $w$ is defined by the cusp function $\gamma(t)$, corresponding to the trace theorem \cite{GV} (see Section 2). Note that this result holds for cusp functions $\gamma(t) = t^{\alpha}$ for all $1 < \alpha < \infty$.

Thus, we can conclude that the unweighted Steklov problem does not have a discrete spectrum in strong outward cuspidal domains \cite{NT}. However, it does have a discrete spectrum when we consider the weighted Steklov problem with weights corresponding to the geometry of the cusp.

In the second part of the article, we consider the Steklov type $p$-eigenvalue problem for $1 < p < \infty$:
\begin{equation}
\label{steqn}
\begin{cases}
-\Delta_{p} u+|u|^{p-2}u&=0\text{ in }\Omega_{\gamma},\\
|\nabla u|^{p-2}\nabla u\cdot\nu&=\lambda w |u|^{p-2}u\text{ on }\partial\Omega_{\gamma}.
\end{cases}
\end{equation}
We remark that equation \eqref{steqn} can be linear ($p=2$) or non-linear ($p \neq 2$). Such problems represent the Schr\"odinger--Steklov eigenvalue problem. Moreover, they are connected with the Sobolev trace inequality: there exists a constant $S>0$ such that the inequality
\begin{equation}
\label{eqn_s}
S^{\frac{1}{p}} \|u\|_{L^p_w(\partial \Omega_{\gamma})} \leq \|u\|_{W^{1,p}(\Omega_{\gamma})}
\end{equation}
holds for any function $u\in W^{1,p}(\Omega_{\gamma})$.

The optimal constant $S$ in the inequality $($\ref{eqn_s}$)$ coincides with the principal eigenvalue of the associated Schr\"odinger--Steklov problem.

We suggest an approach based on the compactness of the trace embedding operators of Sobolev spaces into weighted Lebesgue spaces with weights associated with the cusp function of outward cuspidal domains. By using the compactness of the trace embedding operator \cite{GV},
$$
i: W^{1,p}(\Omega_{\gamma}) \hookrightarrow L^p_w(\partial \Omega_{\gamma}),
$$
we consider the Rayleigh--Steklov quotient:
$$
R(u) = \frac{\int_{\Omega_{\gamma}} (|\nabla u|^p + |u|^p) \, dx}{\int_{\partial \Omega_{\gamma}} |u|^p w \, ds}.
$$
Using this Rayleigh--Steklov quotient, we prove the variational characterization of the weighted Steklov eigenvalues in outward cuspidal domains $\Omega_{\gamma} \subset \mathbb{R}^n$.

In the final part of the article, we use the inverse iteration method to demonstrate the existence of a non-increasing sequence of eigenvalues for the non-linear problem. In addition, we establish a convergence result for the corresponding sequence of eigenfunctions.

The paper is organized as follows: In Section 2, we discuss the functional setting. In Section 3, we study the weighted linear Steklov problem and the weighted linear Schr\"odinger--Steklov problem. Section 4 is devoted to the non-linear weighted Schr\"odinger--Steklov problem. Finally, in Section 5, we establish existence results for the weighted Steklov $p$-eigenvalue problem by using the inverse iteration method.

\section{Functional setting}

Let us recall the basic notions of the Sobolev spaces.
Let $\Omega$ be an open subset of $\mathbb R^n$. The Sobolev space $W^{1,p}(\Omega)$, $1<p<\infty$, is defined \cite{M}
as a Banach space of locally integrable weakly differentiable functions
$u:\Omega\to\mathbb{R}$ equipped with the following norm:
$$
\|u\|_{W^{1,p}(\Omega)}=\left(\int_{\Omega}|\nabla u(x)|^p\,dx+\int_{\Omega}|u(x)|^p\,dx\right)^\frac{1}{p},
$$
where $\nabla u$ is the weak gradient of the function $u$, i.~e. $ \nabla u = (\frac{\partial u}{\partial x_1},...,\frac{\partial u}{\partial x_n})$.

The following result, can be found, for example in \cite[Proposition~9.1]{Br10}, \cite[Paragraph~1.4]{IC} and \cite{M}.
\begin{lemma}
\label{Xuthm}
The space $W^{1,p}(\Omega)$, $1<p<\infty$, is real separable and uniformly convex Banach space.
\end{lemma}

Let $E\subset\mathbb R^n$ be a Borel set E. Then $E$ is said to be $H^{m}$-rectifiable set \cite{Fe69}, if $E$ is of Hausdorff dimension $m$, and there exists a countable collection $\{\varphi_i\}_{i\in\mathbb{N}}$ of Lipschitz continuous mappings
$$
\varphi_i: \mathbb R^m\to\mathbb R^n, 
$$   
such that the $m$-Hausdorff measure $H^m$ of the set $E\setminus \bigcup_{i=1}^{\infty}\varphi_i(\mathbb R^m)$
is zero. 

Let $\Omega\subset\mathbb R^n$ be a domain with $H^{n-1}$-rectifiable boundary $\partial\Omega$ and $w:\partial\Omega\to\mathbb R$ be a non-negative continuous function. We consider the weighted Lebesgue space $L^{p}_w(\partial\Om)$ with the following norm
$$
\|u\|_{L^p_w(\partial\Omega)}=\left(\int_{\partial\Om}|u(x)|^p w(x)\,ds(x)\right)^\frac{1}{p},
$$
where $ds$ is the $(n-1)$-dimensional surface measure on $\partial\Omega$.

In accordance with the outward cuspidal domain $\Omega_{\gamma}$ defined by (\ref{cusp}), we define a continuous weight function $w:\partial\Omega_{\gamma}\to\mathbb R$ setting 
\begin{equation}
\label{weight}
w(x_1,...,x_{n-1},x_n)=
\begin{cases}
\gamma(x_n),\,\,&\text{if}\,\, \sqrt{x_1^2+...+x_{n-1}^2}=\gamma(x_n)<1,
\\
1,\,\,&\text{if}\,\, \sqrt{x_1^2+...+x_{n-1}^2}=\gamma(x_n)\geq 1.
\end{cases}
\end{equation}

The  following theorem is the direct consequence of \cite[Theorem 2.3]{GV} and the fact, that domains of the class $OP_{\varphi}$, which are considered in \cite{GV}, are bi-Lipschitz equivalent to outward cuspidal domains $\Omega_{\gamma}$.

\begin{theorem}
\label{trace}
Let $\Omega_{\gamma}$ be an outward cuspidal domain defined by (\ref{cusp}) and the weight $w$ be defined by (\ref{weight}). Then the trace embedding operator 
$$
i: W^{1,p}(\Omega_{\gamma})\hookrightarrow L^p_w(\partial\Omega_{\gamma})
$$
is compact.
\end{theorem}

\section{Linear weighted eigenvalue problem}
\subsection{The linear weighted Steklov eigenvalue problem} Let $\Omega_{\gamma}$ be an outward cuspidal domain defined by (\ref{cusp}) and the weight $w$ be defined by (\ref{weight}). We consider in $\Omega_{\gamma}$ the weighed Steklov linear eigenvalue problems given by \eqref{steq} which reads as,
\begin{equation}
\label{WEn}
\begin{cases}
-\Delta u=0\text{ in }\Omega_{\gamma},\\
\nabla u\cdot\nu=\lambda w u\text{ on }\partial\Omega_{\gamma}.
\end{cases}
\end{equation}

\begin{definition}
\label{def_S}
We say that $(\lambda,u)\in \R\times (W^{1,2}(\Omega_{\gamma})\setminus\{0\})$ is an eigenpair of \eqref{WEn} if for every function $v\in W^{1,2}(\Omega_{\gamma})$, we have
\begin{equation}
\label{skwksol_S}
\begin{split}
\int_{\Omega_{\gamma}}\nabla u\nabla v\,dx=\lambda \int_{\partial \Omega_{\gamma}}u v\,wds(x).
\end{split}
\end{equation}
We refer to $\lambda$ as an eigenvalue and $u$ as an eigenfunction of \eqref{WEn} corresponding to the eigenvalue $\lambda$.
\end{definition}

The main result of this subsection reads as follows:

\begin{theorem}\label{dthm1}
Let $\Omega_{\gamma}$ be an outward cuspidal domain defined by \eqref{cusp}, and let the weight $w$ be defined by \eqref{weight}. Then the spectrum of the weighted Steklov eigenvalue problem \eqref{WEn} is discrete and is given by a non-decreasing sequence
\[
0 = \lambda_0 < \lambda_1 \leq \lambda_2 \leq \dots \leq \lambda_k \leq \dots\,,
\]
where each eigenvalue is repeated according to its finite algebraic multiplicity, and $\lambda_k \to \infty$ as $k \to \infty$.
\end{theorem}

\begin{proof}
The weight function $w\in L^{\infty}(\partial \Omega_{\gamma})$ because it is a continuous function on the compact $\partial \Omega_{\gamma}$. 
Therefore, the weighted Steklov eigenvalue problem \eqref{WEn} corresponds to the variational triple $(D, a, b)$ in the sense of \cite[Section 1.10]{RSS},  where
$$
D=\left\{u\in W^{1,2}(\Omega_{\gamma})\,: \int\limits_{\partial \Omega_{\gamma}}u(x)w(x)~ds(x)=0\right\};
$$
and the quadratic forms are given by
$$
a[u]=\int\limits_{\Omega_{\gamma}}|\nabla u(x)|^2~dx,\,\, b[u]=\int\limits_{\partial \Omega_{\gamma}}|u(x)|^2 w(x)~ds(x).
$$

By Theorem~\ref{trace} the trace embedding operator $i: W^{1,p}(\Omega_{\gamma})\hookrightarrow L^p_w(\partial\Omega_{\gamma})
$ is bounded. The boundedness of the trace operator ensures that the bilinear quadratic form $b[\cdot,\cdot]$ is well-defined on $D$, and thus Friedrich's theorem applies (see \cite[Theorem 1.5]{RSS}) with the target Hilbert space $L^2_w(\partial\Omega_{\gamma})$.
Hence we can define the positive self-adjoint operator  corresponding to the weighted Steklov eigenvalue problem
$$
S:  W^{1,2}(\Omega_{\gamma})\to L^2_w(\partial\Omega_{\gamma})
$$
by the rule \cite[Formula (1.10)]{RSS}
$$
a[Su,v]=b[u,v],
$$
that means
$$
\int\limits_{\Omega_{\gamma}} \nabla (Su(x))\cdot \nabla v~dx=\int\limits_{\partial \Omega_{\gamma}} u(x)v(x)~ w(x)ds(x), \,\,\text{for any}\,\, v\in D,\,\, u\in D.
$$

By \cite[Lemma 5.1]{RSS} the operator $S$ has compact resolvent if and only
if the the trace embedding operator
$$
i: W^{1,2}(\Omega_{\gamma})\hookrightarrow L^2_w(\partial\Omega_{\gamma})
$$
is compact.

Hence by Theorem~\ref{trace} we obtain that  the spectrum of  the weighted eigenvalue problem (\ref{WEn}) is discrete and can be written in the form of a non-decreasing sequence
\[
0=\lambda_0<\lambda_{1}\leq\lambda_{2}\leq...\leq\lambda_{k}\leq...\,,
\]
where each eigenvalue is repeated as many times as its multiplicity. 
\end{proof}

By Theorem~\ref{dthm1} and the spectral theory of self-adjoint linear operators \cite{Dav}, we have also the following properties for the spectrum of the weighted Steklov eigenvalue problem $($\ref{WEn}$)$:

\begin{corollary}
Let $\Omega_{\gamma}$ be an outward cuspidal domain defined by \eqref{cusp}, and let the weight $w$ be defined by \eqref{weight}. Then the spectrum of the weighted Steklov eigenvalue problem \eqref{WEn} has the following properties:

\noindent
(i) the limit
\[
\lim\limits _{k\to\infty}\lambda_{k}=\infty\,,
\]

\noindent
(ii) for each $k\in\mathbb{N}$, the Min-Max Principle
\begin{equation}
\lambda_{n}=\inf\limits _{\substack{L\subset W^{1,2}(\Omega_{\gamma})\\
\dim L=n
}
}\,\,\sup\limits _{\substack{u\in L\\
u\ne 0
}
}\frac{\int_{\Omega_{\gamma}}|\nabla u|^{2}~dx}{\int_{\partial \Omega_{\gamma}}|u|^{2}w~ds}
\label{MinMax}
\end{equation}
holds, and
\begin{equation}
\lambda_{n}=\sup\limits _{\substack{u\in M_{n}\\
u\ne 0
}
}\frac{\int_{\Omega_{\gamma}}|\nabla u|^{2}~dx}{\int_{\partial\Omega_{\gamma}}|u|^{2}w~ds}\label{MaxPr}
\end{equation}
 where
\[
M_{n}={\rm span}\,\{v_{1},v_{2},...v_{n}\}
\]
and $\{v_{k}\}_{k\in\mathbb{N}}$ is an orthonormal (in the space $W^{1,2}(\Omega_{\gamma})$) set of eigenfunctions
corresponding to the eigenvalues $\{\lambda_{k}\}_{k\in\mathbb{N}}$.
\end{corollary}

\subsection{The linear weighted Schr\"odinger--Steklov eigenvalue problem}

Let $\Omega_{\gamma}$ be an outward cuspidal domain defined by (\ref{cusp}) and let the weight $w$ be defined by (\ref{weight}). 
We consider the \emph{reduced} linear weighted Schr\"odinger--Steklov eigenvalue problem in $\Omega_{\gamma}$:
\begin{equation}
\label{WeP}
\begin{cases}
-\Delta u + u = 0 & \text{in } \Omega_{\gamma},\\
\nabla u \cdot \nu = \lambda\, w\, u & \text{on } \partial\Omega_{\gamma},
\end{cases}
\end{equation}
together with the orthogonality condition
\begin{equation}
\label{orth}
\int_{\partial\Omega_{\gamma}} u(x)\, w(x)\, ds(x) = 0.
\end{equation}

Next, we define the notion of a weak solution of the problem \eqref{WeP}--\eqref{orth}.

\begin{definition}
\label{defL}
We say that $(\lambda,u)\in \mathbb{R} \times (W^{1,2}(\Omega_{\gamma})\setminus\{0\})$ 
is an eigenpair of \eqref{WeP}--\eqref{orth} if for every function 
$v\in W^{1,2}(\Omega_{\gamma})$ we have
\begin{equation}
\label{skwksol}
\int_{\Omega_{\gamma}} \nabla u \cdot \nabla v\,dx
+ \int_{\Omega_{\gamma}} u\, v\,dx
= \lambda \int_{\partial \Omega_{\gamma}} u\, v\, w\,ds(x),
\end{equation}
and $u$ satisfies the orthogonality condition \eqref{orth}.
\end{definition}
We refer to $\lambda$ as an eigenvalue and $u$ as an eigenfunction of \eqref{WeP}--\eqref{orth} corresponding to the eigenvalue $\lambda$.

The main result of this subsection reads as follows:

\begin{theorem}
Let $\Omega_{\gamma}$ be an outward cuspidal domain defined by \eqref{cusp}, and let the weight $w$ be defined by \eqref{weight}. Then the spectrum of the reduced problem \eqref{WeP}--\eqref{orth} is discrete and is given by a non-decreasing sequence
\[
0 < \lambda_0 \leq \lambda_1 \leq \lambda_2 \leq \dots \leq \lambda_k \leq \dots\,,
\]
where each eigenvalue is repeated according to its finite algebraic multiplicity, and $\lambda_k \to \infty$ as $k \to \infty$.
\end{theorem}

\begin{proof}
The weighted Schr\"odinger--Steklov eigenvalue problem \eqref{WeP} corresponds to the variational triple $(D, a, b)$ in the sense of \cite[Section 1.10]{RSS}, where 
$$
D=\left\{u\in W^{1,2}(\Omega_{\gamma})\,: \int\limits_{\partial \Omega_{\gamma}}u(x)w(x)~ds(x)=0\right\};
$$
and the quadratic forms are given by
$$
a[u]=\int\limits_{\Omega_{\gamma}}|\nabla u(x)|^2~dx+\int\limits_{\Omega_{\gamma}}|u(x)|^2~dx,\,\, b[u]=\int\limits_{\partial \Omega_{\gamma}}|u(x)|^2 w(x)~ds(x).
$$

Taking into account Theorem~\ref{trace}, which states the compactness of the trace embedding operator
$$
i: W^{1,2}(\Omega_{\gamma})\hookrightarrow L^2_w(\partial\Omega_{\gamma})
$$ and using \cite[Theorem 1.5]{RSS} as in Theorem \ref{dthm1} above, the result follows.
\end{proof}

\section{Weighted Steklov $p$-eigenvalue problems}

Let $\Omega_{\gamma}$ be an outward cuspidal domain defined by (\ref{cusp}), and let the weight function $w$ be given by (\ref{weight}). 
We consider the \emph{reduced} weighted Schr\"odinger--Steklov $p$-eigenvalue problem, for $1 < p < \infty$:
\begin{equation}
\label{WePN}
\begin{cases}
-\mathrm{div}(|\nabla u|^{p-2}\nabla u) + |u|^{p-2}u = 0 & \text{in } \Omega_{\gamma},\\
|\nabla u|^{p-2} \nabla u \cdot \nu = \lambda\, w\, |u|^{p-2}u & \text{on } \partial\Omega_{\gamma},
\end{cases}
\end{equation}
together with the orthogonality condition
\begin{equation}
\label{orth-p}
\int_{\partial \Omega_{\gamma}} |u|^{p-2} u\, w\, ds = 0.
\end{equation}

\begin{definition}
\label{def-p}
We say that $(\lambda,u)\in \mathbb{R} \times (W^{1,p}(\Omega_{\gamma})\setminus\{0\})$ 
is an eigenpair of \eqref{WePN}--\eqref{orth-p} if for every $v\in W^{1,p}(\Omega_{\gamma})$ we have
\begin{equation}
\label{skwksolN}
\int_{\Omega_{\gamma}} |\nabla u|^{p-2} \nabla u \cdot \nabla v\,dx
+ \int_{\Omega_{\gamma}} |u|^{p-2} u\, v\,dx
= \lambda \int_{\partial \Omega_{\gamma}} |u|^{p-2} u\, v\, w\,ds(x),
\end{equation}
and $u$ satisfies the orthogonality condition \eqref{orth-p}.
\end{definition}
We refer to $\lambda$ as an eigenvalue and $u$ as an eigenfunction of \eqref{WePN}--\eqref{orth-p} corresponding to the eigenvalue $\lambda$.

The equation \eqref{WePN} represents the Euler-Lagrange equation corresponding, in its weak formulation \eqref{skwksolN}, to the functional
$$
F = \|\nabla v\|_{L^p(\Omega_{\gamma})}^p + \|v\|_{L^p(\Omega_{\gamma})}^p,
$$
restricted to the set
$$
S = \left\{ u \in W^{1,p}(\Omega_{\gamma}) : \|u\|_{L^p_w(\partial \Omega_{\gamma})} = 1 \right\}.
$$

The following theorem provides the existence and variational characterization of the first non-trivial eigenvalue \( \lambda_p \) associated with the weighted Schr\"odinger-Steklov \( p \)-eigenvalue problem, described in terms of the minimum of the Rayleigh quotient. The orthogonality condition
$$
\int_{\partial \Omega_{\gamma}} |u|^{p-2} u\, w\, ds = 0
$$
ensures that the eigenfunction is non-trivial and plays a key role in isolating the first non-zero eigenvalue.

\begin{theorem}\label{minthm}
Let $\Omega_{\gamma}$ be an outward cuspidal domain defined by (\ref{cusp}) and the weight $w$ be defined by (\ref{weight}). Then 
for the reduced problem \eqref{WePN}--\eqref{orth-p}, $1<p<\infty$, there exists $u\in W^{1,p}(\Omega_{\gamma})\setminus\{0\}$ satisfying \eqref{orth-p}. Moreover, the first non-trivial eigenvalue $\lambda_{p}$ is given by
\begin{multline}
\label{min}
\lambda_{p}=
\inf \left\{\frac{\|\nabla v\|_{L^p(\Omega_{\gamma})}^p+\|v\|_{L^p(\Omega_{\gamma})}^p}{\|v\|_{L^p_w(\partial \Omega_{\gamma})}^p} : v \in W^{1,p}(\Omega_{\gamma}) \setminus \{0\},
\int_{\partial \Omega_{\gamma}} |v|^{p-2}v w\,ds=0 \right\}\\
=\frac{\|\nabla u\|_{L^p(\Omega_{\gamma})}^p+\|u\|_{L^p(\Omega_{\gamma})}^p}{\|u\|_{L^p_w(\partial \Omega_{\gamma})}^p}.
\end{multline}
\end{theorem}

\begin{proof} 
Note that if the weighted boundary norm $\|v\|_{L^p_w(\partial \Omega_\gamma)}$ vanishes for some admissible function $v$, then the Rayleigh quotient is considered infinite, and such functions do not affect the value of the infimum.

By Theorem~\ref{trace} the trace operator 
$$
i: W^{1,p}(\Omega_{\gamma})\hookrightarrow L^p_w(\partial\Omega_{\gamma})
$$
is well defined, and we can define the functional $G:W^{1,p}(\Omega_{\gamma})\to\mathbb{R}$ by
$$
G(v)=\int_{\partial\Omega_{\gamma}}|v|^{p-2}vw\,ds.
$$
For $k\in\mathbb{N}$, we define
$
H_\frac{1}{k}: W^{1,p}(\Omega_{\gamma})\to\mathbb{R}
$
by
$$
H_{\frac{1}{k}}(v)=\|v\|^{p}_{L^p(\Omega_{\gamma})}+\|\nabla v\|^{p}_{L^p(\Omega_{\gamma})}-\Big(\lambda_p+\frac{1}{k}\Big)\|v\|^{p}_{L^p_{w}(\partial\Omega_{\gamma})}.
$$
By the definition of infimum, which defined in (\ref{min}), for every $k\in\mathbb{N}$, there exists a function $u_k\in W^{1,p}(\Omega_{\gamma})\setminus\{0\}$ such that
$$
\int_{\partial\Omega_{\gamma}}|u_k|^{p-2}u_k w\,dx=0\text{ and }H_{\frac{1}{k}}(u_k)<0.
$$
Without loss of generality, we assume that $\|u_k\|^{p}_{L^p(\Omega_{\gamma})}+\|\nabla u_k\|^{p}_{L^p(\Omega_{\gamma})}=1$. Therefore, the sequence $\{u_k\}_{k\in\mathbb{N}}$ is uniformly bounded in $W^{1,p}(\Omega_{\gamma})$. Hence, by Theorem \ref{trace}, there exists $u\in W^{1,p}(\Omega_{\gamma})$ such that $u_k\rightharpoonup  u$ weakly in $W^{1,p}(\Omega_{\gamma})$, $u_k\to u$ strongly in $L^p_{w}(\partial\Omega_{\gamma})$ and $\nabla u_k\rightharpoonup   \nabla u$ weakly in $L^p(\Omega_{\gamma})$. Moreover, by \cite[Theorem 4.9]{Br10}, there exists $g\in L_{w}^p(\partial\Omega_{\gamma})$ such that $|u_k|\leq g$ for $H^{n-1}$ almost everywhere on $\partial\Omega_{\gamma}$. Hence, by the Lebesgue's dominated convergence theorem, we have
$$
\int_{\partial\Omega_{\gamma}}|u|^{p-2}u w\,ds=\lim_{k\to\infty}\int_{\partial\Omega_{\gamma}}|u_k|^{p-2}u_k w\,ds=0.
$$
Since $H_\frac{1}{k}(u_k)<0$, we have
\begin{equation}\label{lim}
\|u_k\|^{p}_{L^p(\Omega_{\gamma})}+\|\nabla u_k\|^{p}_{L^p(\Omega_{\gamma})}-\Big(\lambda_p+\frac{1}{k}\Big)\|u_k\|^{p}_{L^p_{w}(\partial\Omega_{\gamma})}<0.
\end{equation}
Moreover, since $u_k\rightharpoonup  u$ weakly in $W^{1,p}(\Omega_{\gamma})$ and $\nabla u_k\rightharpoonup  \nabla u$ weakly in $L^p(\Omega_{\gamma})$, we get
\begin{equation}\label{wlsc}
\|u\|^{p}_{L^p(\Omega_{\gamma})}+\|\nabla u\|^{p}_{L^p(\Omega_{\gamma})}\leq\lim\inf_{k\to\infty}\Big(\|u_k\|^{p}_{L^p(\Omega_{\gamma})}+\|\nabla u_k\|^{p}_{L^p(\Omega_{\gamma})}\Big).
\end{equation}
Using \eqref{wlsc} and Theorem \ref{trace} and then passing the limit in (\ref{lim}), we have
$$
\lambda_p\geq \frac{ \|u\|^{p}_{L^p(\Omega_{\gamma})}+\|\nabla u\|^{p}_{L^p(\Omega_{\gamma})} }{\|u\|^p_{L^p_{w}(\partial\Omega_{\gamma})}}.
$$
This combined with the definition of $\lambda_p$, we obtain
$$
\lambda_p=\frac{ \|u\|^{p}_{L^p(\Omega_{\gamma})}+\|\nabla u\|^{p}_{L^p(\Omega_{\gamma})} }{\|u\|^p_{L^p_{w}(\partial\Omega_{\gamma})}}.
$$

Again, since $H_\frac{1}{k}(u_k)<0$ and $\|u_k\|^{p}_{L^p(\Omega_{\gamma})}+\|\nabla u_k\|^{p}_{L^p(\Omega_{\gamma})}=1$, we have
$$
1-\Big(\lambda_p+\frac{1}{k}\Big)\|u_k\|_{L^p_{w}(\partial\Omega_{\gamma})}^p<0.
$$
Letting $k\to\infty$, we get
$$
\|u\|^p_{L^{p}_{w}(\partial\Omega_{\gamma})}\lambda_p\geq 1,
$$
which gives $\lambda_p>0$ and $u\neq 0$ almost everywhere in $\Omega_\gamma$.
\end{proof}

\begin{theorem}
\label{sol1}
Let $\Omega_{\gamma}$ be an outward cuspidal domain defined by (\ref{cusp}) and the weight $w$ be defined by (\ref{weight}). Then 
for the weighted Schr\"odinger--Steklov $p$-eigenvalue problem, $1<p<\infty$, there exists a sequence
of eigenvalues $\{\lambda_k\}_{k\in\mathbb{N}}$ of the problem \eqref{WePN}-\eqref{orth-p} such that $\lambda_k\to+\infty$ as $k\to +\infty$.
\end{theorem}

\begin{proof}
Taking into account Theorem \ref{trace}, the proof follows along the lines of the proof of \cite[Theorem 1.3]{BR}. For convenience of the reader, we present few important details below that are crucial to deal with our weighted structure. To this end, as in \cite[page 207-208]{BR}, for any $\alpha>0$, we define the set $S_\alpha=\{u\in W^{1,p}(\Omega_{\gamma}): \|u\|_{W^{1,p}(\Omega_{\gamma})}^p=p\alpha\}$ and
$$
\phi(u)=\frac{1}{p}\int_{\Omega_{\gamma}}|u|^p w\,ds.
$$
Further, we define $\rho:W^{1,p}(\Omega_{\gamma})\setminus\{0\}\to (0,\infty)$ by
$$
\rho(u)=\left(\frac{p\alpha} {\|u\|^p_{ W^{1,p}(\Omega_{\gamma})} }\right)^\frac{1}{p}.
$$

Let $(W^{1,p}(\Omega_{\gamma}))^*$ denote the dual space of $W^{1,p}(\Omega_{\gamma})$ and we define $J:(W^{1,p}(\Omega_{\gamma}))^*\to W^{1,p}(\Omega_{\gamma})$ as the duality mapping such that for any given $\psi\in (W^{1,p}(\Omega_{\gamma}))^*$, there exists a unique element in $W^{1,p}(\Omega_{\gamma})$, say $J(\psi)$ satisfying
$$
\langle \psi,J(\psi)\rangle=\|\psi\|^2_{W^{1,p}(\Omega_{\gamma})^*}
$$
and 
$$
\|J(\psi)\|_{W^{1,p}(\Omega_{\gamma})}=\|\psi\|_{(W^{1,p}(\Omega_{\gamma}))^*}.
$$

Now, we define 
$$
T u=J(Du)-A(u),\,\, u\in W^{1,p}(\Omega_{\gamma}),
$$
where
$$
\langle Du;v\rangle=\int_{\partial\Omega_{\gamma}}|u|^{p-2}uv w\,ds-\langle Pu;v\rangle,
$$

$$
\langle Pu;v\rangle=\frac{\int_{\partial\Omega_{\gamma}}|u|^p w\,ds }{\|u\|^p_{W^{1,p}(\Omega_{\gamma})}}\Big(\int_{\Omega_{\gamma}}(|\nabla u|^{p-2}\nabla u\nabla v+|u|^{p-2}uv)\,dx - \int_{\partial\Omega_{\gamma}}|\nabla u|^{p-2}\frac{\partial u}{\partial\nu}v\,ds\Big),
$$
and
$$
A=\frac{\langle \rho'(u);J(Du)\rangle \langle Pu+Du;u\rangle+\langle Pu; J(Du)\rangle} {\big(\langle \rho'(u);u\rangle+1\big)\langle Pu+Du;u\rangle}.
$$
Now taking into account the above mappings along with Theorem \ref{trace}, the result follows along the lines of the proof of \cite[Theorem 1.3]{BR}.
\end{proof}

\section{Existence results for weighted Steklov $p$-eigenvalue problems by inverse iteration method}

In this section, we establish existence results for the weighted Steklov $p$-eigenvalue problem defined in \eqref{WePN}. We recall that $\Omega_{\gamma}$ is an outward cuspidal domain defined by (\ref{cusp}) and the weight $w$ be defined by (\ref{weight}).

Before stating our main theorems below, we rewrite the definition \eqref{min} of the first non-trivial eigenvalue $\lambda_{p}$ in the following equivalent form:
\begin{equation}\label{la}
\lambda_p:=\inf_{\{u\in W^{1,p}(\Omega_{\gamma})\cap {L^p_{w}(\partial\Omega_{\gamma})},\,\,\|u\|_{L^p_{w}(\partial\Omega_{\gamma})}=1\}}\int_{\Omega_{\gamma}}\Big(|\nabla u|^p+|u|^p\Big)\,dx.
\end{equation}

\begin{theorem}\label{newthm}
Suppose $1 < p < \infty$. Then the following properties hold:
\vskip 0.2cm
\noindent
\textup{(a)} There exists a sequence $\{w_n\}_{n \in \mathbb{N}} \subset W^{1,p}(\Omega_{\gamma}) \cap L^p_{w}(\partial\Omega_{\gamma})$ such that $\|w_n\|_{L^p_{w}(\partial\Omega_{\gamma})} = 1$ for all $n$, and for every $v \in W^{1,p}(\Omega_{\gamma})$, the following identity holds:
\begin{equation}\label{its}
\int_{\Omega_{\gamma}} |\nabla w_{n+1}|^{p-2} \nabla w_{n+1}  \nabla v\, dx + \int_{\Omega_{\gamma}} |w_{n+1}|^{p-2} w_{n+1} v\, dx = \mu_n \int_{\partial\Omega_{\gamma}} |w_n|^{p-2} w_n v w\, ds,
\end{equation}
where
\begin{equation}\label{its1}
\mu_n \geq \lambda_p,
\end{equation}
with $\lambda_p$ defined in \eqref{la}.
\vskip 0.2cm
\noindent
\textup{(b)} The sequences $\{\mu_n\}_{n \in \mathbb{N}}$ and $\{\|w_{n+1}\|_{W^{1,p}(\Omega_{\gamma})}^p\}_{n \in \mathbb{N}}$ are non-increasing and converge to the same limit $\mu \geq \lambda_p$.
\vskip 0.2cm
\noindent
\textup{(c)} There exists a subsequence $\{n_j\}_{j \in \mathbb{N}}$ such that both $\{w_{n_j}\}$ and its forward shift $\{w_{n_{j+1}}\}$ converge strongly in $W^{1,p}(\Omega_{\gamma})$ to the same limit $w \in W^{1,p}(\Omega_{\gamma}) \cap L^p_{w}(\partial\Omega_{\gamma})$, with $\|w\|_{L^p_{w}(\partial\Omega_{\gamma})} = 1$. Moreover, the pair $(\mu, w)$ satisfies the eigenvalue problem \eqref{WePN}.
\end{theorem}

\begin{theorem}
\label{subopthm1}
Let $1 < p < \infty$. Suppose $\{u_n\}_{n \in \mathbb{N}} \subset W^{1,p}(\Omega_{\gamma}) \cap L^p_{w}(\partial\Omega_{\gamma})$ is a sequence such that $\|u_n\|_{L^p_{w}(\partial\Omega_{\gamma})} = 1$ for all $n$, and
\[
\lim_{n \to \infty} \|u_n\|_{W^{1,p}(\Omega_{\gamma})}^p = \lambda_p,
\]
where $\lambda_p$ is defined in \eqref{la}.
\vskip 0.2cm
\noindent
\textup{(a)} Then there exists a subsequence $\{u_{n_j}\}_{j \in \mathbb{N}}$ that converges strongly in $W^{1,p}(\Omega_{\gamma})$ to a function $u \in W^{1,p}(\Omega_{\gamma}) \cap L^p_{w}(\partial\Omega_{\gamma})$ with $\|u\|_{L^p_{w}(\partial\Omega_{\gamma})} = 1$, and
\[
\lambda_p = \int_{\Omega_{\gamma}} |\nabla u|^p\, dx + \int_{\Omega_{\gamma}} |u|^p\, dx.
\]
\vskip 0.2cm
\noindent
\textup{(b)} Moreover, $(\lambda_p, u)$ is an eigenpair of \eqref{WePN}, and every eigenfunction associated with $\lambda_p$ is a scalar multiple of such limit functions at which $\lambda_p$ is attained.
\end{theorem}

\subsection{Auxiliary results}
In this subsection, we prove some auxiliary results that are needed to prove Theorem \ref{newthm} and Theorem \ref{subopthm1} above. These results mainly follow by using the inverse iteration method introduced in \cite{Ercole}.  We begin by stating the following result from \cite[Theorem $9.14$]{var}:
\begin{theorem}\label{MB}
Let $V$ be a real separable reflexive Banach space and $V^*$ be the dual of $V$. Assume that $A:V\to V^{*}$ is a bounded, continuous, coercive and monotone operator. Then $A$ is surjective, i.e., given any $f\in V^{*}$, there exists $u\in V$ such that $A(u)=f$. If $A$ is strictly monotone, then $A$ is also injective. 
\end{theorem}

First, we provide the preliminaries related to the functional properties of operators defined by the problem \eqref{WePN}.
We define the operators $A : W^{1,p}(\Omega_{\gamma}) \to \left(W^{1,p}(\Omega_{\gamma})\right)^*$ by
\begin{equation}\label{a}
\begin{split}
\langle A(f),v\rangle&=\int_{\Omega_{\gamma}}|\nabla f|^{p-2}\nabla f\nabla v\,dx+\int_{\Omega_{\gamma}}|f|^{p-2}fv\,dx\quad \forall v\,\in W^{1,p}(\Omega_{\gamma})
\end{split}
\end{equation}
and $B:L^p_{w}(\partial\Omega_{\gamma})\to (L^p_{w}(\partial\Omega_{\gamma}))^*$ by
\begin{equation}\label{b}
\begin{split}
\langle B(f),v\rangle=\int_{\partial\Omega_{\gamma}}|f|^{p-2}fv\,w\,ds,\quad \forall v\,\in L^p_{w}(\partial\Omega_{\gamma}).
\end{split}
\end{equation}
The symbols ${(W^{1,p}(\Omega_{\gamma}))}^*$ and $(L^p_{w}(\partial\Omega_{\gamma}))^*$ denotes the dual of $W^{1,p}(\Omega_{\gamma})$ and $L^p_{w}(\partial\Omega_{\gamma})$ respectively. First, we have the following result.
\begin{lemma}\label{newlem}
$(i)$ The operators $A$ defined by \eqref{a} and $B$ defined by \eqref{b} are continuous. $(ii)$ Moreover, $A$ is bounded, coercive and monotone.
\end{lemma}
\begin{proof}
\noindent
$(i)$ \textbf{Continuity:} Suppose $f_n\in W^{1,p}(\Omega_{\gamma})$ such that $f_n\to f$ in the norm of $W^{1,p}(\Omega_{\gamma})$. Thus, up to a subsequence $\nabla f_{n}(x)\to \nabla f(x)$ for almost every $x\in\Omega_{\gamma}$. We observe for $p'=\frac{p}{p-1}$ that
\begin{equation}\label{mfd}
\||\nabla f_n|^{p-2}\nabla f_n\|_{L^{p'}(\Omega_{\gamma})}\leq {c\|\nabla f_{n}\|^{p-1}_{L^p(\Omega_{\gamma})}}\leq c, 
\end{equation}
for some constant $c>0$, which is independent of $n$. Thus, up to a subsequence, we have
\begin{equation}\label{fc}
|\nabla f_{n}|^{p-2}\nabla f_n\rightharpoonup |\nabla f|^{p-2}\nabla f\text{ weakly in }L^{p'}(\Omega_{\gamma}).
\end{equation}
Similarly, we get
\begin{equation}\label{fcnew}
|f_{n}|^{p-2}f_n\rightharpoonup |f|^{p-2}f\text{ weakly in }L^{p'}(\Omega_{\gamma}).
\end{equation}
Thus $A$ is continuous.\\
To prove the continuity of $B$, let $\{f_n\}_{n\in\mathbb{N}}\in L^p_{w}(\partial\Omega_{\gamma})$ converges strongly to $f\in L^p_{w}(\partial\Omega_{\gamma})$. Thus, up to a subsequence $f_n\to f$ for almost every $x\in\Omega_{\gamma}$. We observe that
\begin{equation}\label{cntb}
\||f_n|^{p-2}f_nw^\frac{1}{p'}|\|_{L^{p'}(\partial\Omega_{\gamma})}=\Big(\int_{\partial\Omega_{\gamma}}|f_n|^{p}w\,ds\Big)^\frac{p-1}{p}\leq c,
\end{equation}
for some positive constant $c$ independent of $n$. Hence,
\begin{equation}\label{fcntb}
|f_{n}|^{p-2}f_n\,w^\frac{1}{p'}\rightharpoonup |f|^{p-2}f\,w^\frac{1}{p'}\text{ weakly in }L^{p'}(\Omega_{\gamma}).
\end{equation}
Let $v\in L^p_{w}(\partial\Omega_{\gamma})$. Then $w^\frac{1}{p}v\in L^p(\partial\Omega_{\gamma})$. Therefore, we have
$$
\lim_{n\to\infty}\langle B(f_n),v\rangle=\lim_{n\to\infty}\int_{\partial\Omega_{\gamma}}|f_n|^{p-2}f_nv\,w\,ds=\int_{\partial\Omega_{\gamma}}|f|^{p-2}fv\,w\,ds,
$$
which proves that $B$ is continuous.

\vskip 0.2cm
\noindent
$(ii)$ \textbf{Boundedness:}
First using the Cauchy-Schwarz inequality and then by H\"older's inequality with exponents $p'$ and $p$, for every $f,v\in W^{1,p}(\Omega_{\gamma})$, we obtain
\begin{multline}\label{stest}
%\begin{split}
\langle A(f),v\rangle\\
=\int_{\Omega_{\gamma}}|\nabla f|^{p-2}\nabla f\nabla v\,dx+\int_{\Omega_{\gamma}}|f|^{p-2}fv\,dx
\leq\int_{\Omega_{\gamma}}|\nabla f|^{p-1}|\nabla v|\,dx+\int_{\Omega_{\gamma}}|f|^{p-1}v\,dx\\
\leq\Big(\int_{\Omega_{\gamma}}|\nabla f|^p\,dx\Big)^\frac{p-1}{p}\Big(\int_{\Omega_{\gamma}}|\nabla v|^p\,dx\Big)^\frac{1}{p}+\Big(\int_{\Omega_{\gamma}}|f|^p\,dx\Big)^\frac{p-1}{p}\Big(\int_{\Omega_{\gamma}}|v|^p\,dx\Big)^\frac{1}{p}\\
\leq\Bigg[\Big(\int_{\Omega_{\gamma}}|\nabla f|^p\,dx\Big)^\frac{p-1}{p}+\Big(\int_{\Omega_{\gamma}}|f|^p\,dx\Big)^\frac{p-1}{p}\Bigg]\|v\|_{W^{1,p}(\Omega_{\gamma})}\\
\leq \Big(\int_{\Omega_{\gamma}}|\nabla f|^p\,dx+\int_{\Omega_{\gamma}}|f|^p\,dx\Big)^\frac{p-1}{p}\|v\|_{W^{1,p}(\Omega_{\gamma})}
=\|f\|_{W^{1,p}(\Omega_{\gamma})}^{p-1}\|v\|_{W^{1,p}(\Omega_{\gamma})}.
%\end{split}
\end{multline}
Therefore, we have
$$
\|A(f)\|_{{W^{1,p}(\Omega_{\gamma})}^*}=\sup_{\|v\|_{W^{1,p}(\Omega_{\gamma})}\leq 1}|\langle Af,v\rangle|\leq\|f\|_{W^{1,p}(\Omega_{\gamma})}^{p-1}\|v\|_{W^{1,p}(\Omega_{\gamma})}\leq\|f\|^{p-1}_{W^{1,p}(\Omega_{\gamma})}.
$$
Thus, $A$ is bounded.

\noindent
\textbf{Coercivity:}  We observe that for every $f\in W^{1,p}(\Omega_{\gamma})$,
$$
\langle A(f),f\rangle=\int_{\Omega_{\gamma}}|\nabla f|^p\,dx+\int_{\Omega_{\gamma}}|f|^p\,dx=\|f\|_{W^{1,p}(\Omega_{\gamma})}^p.
$$
Since $p>1$, we have $A$ is coercive.\\

\noindent
\textbf{Monotonicity:}
First, we recall the algebraic inequality from \cite[Lemma 2.1]{Dama}: there exists a constant $C=C(p)>0$ such that
$$
\langle |a|^{p-2}a-|b|^{p-2}b,a-b\rangle\geq C(|a|+|b|)^{p-2}|a-b|^2,
$$
for every $a,b\in\mathbb{R}^N$. Using the above inequality, for every $f,g\in W^{1,p}(\Omega_{\gamma})$, we have
\begin{equation*}
\begin{split}
\langle A(f)-A(g),f-g\rangle&=\int_{\Omega_{\gamma}}\langle |\nabla f|^{p-2}\nabla f-|\nabla g|^{p-2}\nabla g,\nabla (f-g)\rangle\,dx\\
&\quad+\int_{\Omega_{\gamma}}\big(|f|^{p-2}f-|g|^{p-2}g,(f-g)\big)\,dx\geq 0.
\end{split}
\end{equation*}
Thus, $A$ is monotone.
\end{proof}

\begin{lemma}\label{auxlmab}
The operators $A$ defined by \eqref{a} and $B$ defined by \eqref{b} satisfy the following properties: 
\vskip 0.2cm

\noindent
$(H_1)$ $A(tv)=|t|^{p-2}tA(v)\quad\forall t\in\mathbb{R}\quad \text{and}\quad\forall v\in W^{1,p}(\Omega_{\gamma})$.
\vskip 0.2cm

\noindent
$(H_2)$ $B(tv)=|t|^{p-2}tB(v)\quad\forall t\in\mathbb{R}\quad \text{and}\quad\forall v\in L^p_{w}(\partial\Omega_{\gamma})$.
\vskip 0.2cm

\noindent
$(H_3)$ $\langle A(f),v\rangle\leq\|f\|_{W^{1,p}(\Omega_{\gamma})}^{p-1}\|v\|_{W^{1,p}(\Omega_{\gamma})}$ for all $f,v\in W^{1,p}(\Omega_{\gamma})$, where the equality holds if and only if $f=0$ or $v=0$ or $f=t v$ for some $t>0$.
\vskip 0.2cm

\noindent
$(H_4)$ $\langle B(f),v\rangle\leq\|f\|_{L^p_{w}(\partial\Omega_{\gamma})}^{p-1}\|v\|_{L^p_{w}(\partial\Omega_{\gamma})}$ for all $f,v\in {L^p_{w}(\partial\Omega_{\gamma})}$, where the equality holds if and only if $f=0$ or $v=0$ or $f=tv$ for some $t\geq 0$.
\vskip 0.2cm

\noindent
$(H_5)$ For every $f\in L^p_{w}(\partial\Omega_{\gamma})\setminus\{0\}$ there exists $u\in W^{1,p}(\Omega_{\gamma})\setminus\{0\}$ such that
$$
\langle A(u),v\rangle=\langle B(f),v\rangle\quad\forall\quad v\in W^{1,p}(\Omega_{\gamma}).
$$
\end{lemma}
\begin{proof}
\vskip 0.2cm

\noindent
(H1) and (H2) follow directly from the definitions of $A$ and $B$, respectively.
\vskip 0.2cm

\noindent
$(H_3)$ Let $f,v\in W^{1,p}(\Omega_{\gamma})$. Then the inequality $\langle Af,v\rangle\leq \|f\|_{W^{1,p}(\Omega_{\gamma})}^{p-1}\|v\|_{W^{1,p}(\Omega_{\gamma})}$ follows from the proof of boundedness of $A$ in Lemma \ref{newlem} above.\\
If the equality
\begin{equation}\label{stequal}
\langle A(f),v\rangle=\|f\|_{W^{1,p}(\Omega_{\gamma})}^{p-1}\|v\|_{W^{1,p}(\Omega_{\gamma})}
\end{equation}
holds for every $f,v\in W^{1,p}(\Omega_{\gamma})$, we claim that either $f=0$ or $v=0$ or $f=tv$ for some constant $t>0$. Indeed, if $f=0$ or $v=0$, this is trivial. Therefore, we assume $f\neq 0$ and $v\neq 0$ and prove that $f=tv$ for some constant $t>0$.  We observe that if the equality \eqref{stequal} holds, then by the estimate \eqref{stest} we have
\begin{equation}\label{stequal2}
\begin{split}
f_1-f_2=g_2-g_1,
\end{split}
\end{equation}
where
$$
f_1=\int_{\Omega_{\gamma}}|\nabla f|^{p-1}|\nabla v|\,dx,\quad f_2=\Big(\int_{\Omega_{\gamma}}|\nabla f|^p\,dx\Big)^\frac{p-1}{p}\Big(\int_{\Omega_{\gamma}}|\nabla v|^p\,dx\Big)^\frac{1}{p},
$$
$$
g_1=\int_{\Omega_{\gamma}}|f|^{p-1}|v|\,dx,\quad g_2=\Big(\int_{\Omega_{\gamma}}|f|^p\,dx\Big)^\frac{p-1}{p}\Big(\int_{\Omega_{\gamma}}|v|^p\,dx\Big)^\frac{1}{p}.
$$
By H\"older's inequality, we know that $f_1-f_2\leq 0$ and $g_2-g_1\geq 0$. Therefore, we obtain from \eqref{stequal2} that
$$
f_1=f_2\text{ and }g_1=g_2.
$$
Since $g_1=g_2$, the equality in H\"older's inequality holds, which gives
\begin{equation}\label{stCSeq2}
|f(x)|=t|v(x)|\text{ in }\Omega_\gamma,
\end{equation}
for some constant $t>0$.\\
Again, by the estimate \eqref{stest} if the equality \eqref{stequal} holds, then we have
\begin{equation}\label{stequal1}
\langle A(f), v\rangle=\int_{\Omega_{\gamma}}|\nabla f|^{p-1}|\nabla v|\,dx+\int_{\Omega_{\gamma}}|f|^{p-1}|v|\,dx,
\end{equation}
which gives us
\begin{equation}\label{stCS}
\int_{\Omega_{\gamma}}F(x)\,dx+\int_{\Omega_{\gamma}}G(x)\,dx=0,
\end{equation}
where 
$$
F=|\nabla f|^{p-1}|\nabla v|-|\nabla f|^{p-2}\nabla f\nabla v
$$
and
$$
G=|f|^{p-1}|v|-|f|^{p-2}fv.
$$
By Cauchy-Schwarz inequality, we have $F\geq 0$ in $\Omega_{\gamma}$ and $G\geq 0$ in $\Omega_{\gamma}$. Hence using these facts in \eqref{stCS}, we have $G=0$ in $\Omega_{\gamma}$, which reduces to
\begin{equation}\label{stfCS}
f(x)=c(x)v(x)\text{ in }\Omega_{\gamma},
\end{equation}
for some $c(x)\geq 0$ in $\Omega_{\gamma}$.

Combining \eqref{stCSeq2} and \eqref{stfCS}, we get $c(x)=t$ for $x\in\Omega_{\gamma}$ and therefore, we obtain $v=tw$ in $\Omega_{\gamma}$, for some constant $t>0$. Hence, the property $(H_3)$ is verified. 

\vskip 0.2cm

\noindent
$(H_4)$ Let $f,v\in L^p_{w}(\partial\Omega_{\gamma})$. Then first using Cauchy-Schwarz inequality and  then by H\"older's inequality with exponents $p'$ and $p$, we obtain
\begin{multline}
\label{bb}
    %\begin{split}
    \langle B(f),v\rangle=\int_{\partial\Omega_{\gamma}}|f|^{p-2}f\,v\,w\,ds
    \leq \int_{\partial\Omega_{\gamma}}|f|^{p-1}|\,|v|w\,ds\\
    \leq \Big(\int_{\partial\Omega_{\gamma}}|f|^p\,w\,ds\Big)^\frac{p-1}{p}\Big(\int_{\partial\Omega_{\gamma}}|v|^p\,w\,ds\Big)^\frac{1}{p}
    =\|f\|_{L^p_{w}(\partial\Omega_{\gamma})}^{p-1}\|v\|_{L^p_{w}(\partial\Omega_{\gamma})}.
    %\end{split}
\end{multline}

If the equality 
\begin{equation}\label{bbequal}
\langle B(f),v\rangle=\|f\|_{L^p_{w}(\partial\Omega_{\gamma})}^{p-1}\|v\|_{L^p_{w}(\partial\Omega_{\gamma})}
\end{equation}
holds for every $f,v\in L^p_{w}(\partial\Omega_{\gamma})$, we claim that either $f=0$ or $v=0$ or $f=tv$ for some constant $t\geq 0$. Indeed, if $f=0$ or $v=0$, this is trivial. Therefore, we assume $f\neq 0$ and $v\neq 0$ and prove that $f=tv$ for some constant $t\geq 0$.  We observe that if the equality \eqref{bbequal} holds, then by the estimate \eqref{bb} above, we have
\begin{equation}\label{bbequal2}
\begin{split}
\int_{\partial\Omega_{\gamma}}|f|^{p-1}|\,|v|w\,ds=\Big(\int_{\partial\Omega_{\gamma}}|f|^p\,w\,ds\Big)^\frac{p-1}{p}\Big(\int_{\partial\Omega_{\gamma}}|v|^p\,w\,ds\Big)^\frac{1}{p}.
\end{split}
\end{equation}
This means equality in H\"older's inequality holds, which gives 
\begin{equation}\label{bbCSeq2}
|f(x)|=t|v(x)|\text{ on }{\partial\Omega_{\gamma}},
\end{equation}
for some constant $t>0$.\\
Again, by the estimate \eqref{bb} if the equality \eqref{bbequal} holds, then we have
\begin{equation}\label{bbequal1}
\langle B(f), v\rangle=\int_{\partial\Omega_{\gamma}}|f|^{p-1}\,|v|\,w\,ds.
\end{equation}
Hence, the equality in Cauchy-Schwarz inequality holds, which gives us
\begin{equation}\label{bbfCS}
f(x)=c(x)v(x)\text{ on }{\partial\Omega_{\gamma}},
\end{equation}
for some $c(x)\geq 0$ on ${\partial\Omega_{\gamma}}$.

Combining \eqref{bbCSeq2} and \eqref{bbfCS}, we get $c(x)=t$ for $x\in{\partial\Omega_{\gamma}}$ and therefore, we obtain $v=tw$ in ${\partial\Omega_{\gamma}}$, for some constant $t>0$. Hence, the property $(H_4)$ is verified.
\vskip 0.2cm

\noindent
$(H_5)$ Note that by Lemma \ref{Xuthm}, it follows that $W^{1,p}(\Omega_{\gamma})$ is a separable and reflexive Banach space. By Lemma \ref{newlem}, the operator $A:W^{1,p}(\Omega_{\gamma})\to ({W^{1,p}(\Omega_{\gamma})})^*$ is bounded, continuous, coercive and monotone.

By Theorem \ref{trace}, the Sobolev space $W^{1,p}(\Omega_{\gamma})$ is continuously embedded in $L^p_{w}(\partial\Omega_{\gamma})$. Therefore, $B(f)\in ({W^{1,p}(\Omega_{\gamma})})^*$ for every $f\in L^p_{w}(\partial\Omega_{\gamma})\setminus\{0\}$.

Hence, taking into account Lemma \ref{Xuthm}, by Theorem \ref{MB}, for every $f\in L^p_{w}(\partial\Omega_{\gamma})\setminus\{0\}$, there exists $u\in W^{1,p}(\Omega_{\gamma})\setminus\{0\}$ such that
$$
\langle A(u),v\rangle=\langle B(f),v\rangle\quad\forall v\in W^{1,p}(\Omega_{\gamma}).
$$
Hence the property $(H_5)$ holds. This completes the proof.
\end{proof}

\vskip 0.2cm

\textbf{Proof of Theorem \ref{newthm}:}
We begin by recalling the definitions of the operators 
$A : W^{1,p}(\Omega_{\gamma}) \to \left(W^{1,p}(\Omega_{\gamma})\right)^*$ from \eqref{a}, and 
$B : L^p_{w}(\partial\Omega_{\gamma}) \to \left(L^{p}_{w}(\partial\Omega_{\gamma})\right)^*$ from \eqref{b}. 

The proof of part (a) follows by proceeding along the lines of the argument in \cite[pages 579 and 584--585]{Ercole}. 
The proofs of parts (b) and (c) follow similarly from \cite[Lemmas 4 and 5]{Ercole}, respectively. 

For the reader's convenience, we briefly outline the proof of part (c) below.

\vskip 0.2cm
\noindent
$(a)$
We fix $w_0\in L^p_{w}(\partial\Omega_\gamma)$ such that $\|w_0\|_{L^p_{w}(\partial\Omega_\gamma)}=1$. Then by the property $(H_5)$ of Lemma \ref{auxlmab}, it follows that there exists $u_1\in W^{1,p}(\Omega_\gamma)\setminus\{0\}$ such that
$$
\langle A(u_1), v\rangle=\langle B(w_0), v\rangle\quad \forall v\in W^{1,p}(\Omega_\gamma).
$$
We set $w_1=\|u_1\|_{L^p_{w}(\partial\Omega_\gamma)}^{-1}u_1$ and $\mu_1=(\|u_1\|_{L^p_{w}(\partial\Omega_\gamma)})^{1-p}$. By $(H_1)$ and $(H_2)$, multiplying the above equation by $(\|u_1\|_{L^p_{w}(\partial\Omega_\gamma)})^{1-p}$, we obtain
\[
\langle A(w_1), v\rangle = \mu_1\langle B(w_0), v\rangle \quad \forall v \in W^{1,p}(\Omega_\gamma).
\]
Now repeating the above argument, we construct the iterative sequence $\{w_n\}_{n\in\mathbb{N}}\subset W^{1,p}(\Omega_{\gamma})\cap L^p_{w}(\partial\Omega_{\gamma})$ such that \eqref{its} holds, where
$$
\mu_n=(\|u_{n+1}\|_{L^p_{w}(\partial\Omega_\gamma)})^{1-p}
$$
satisfies \eqref{its1}. Indeed, by the definition of $\lambda_p$, we observe that
\begin{multline}
\label{its3}
\lambda_p\leq \|w_{n+1}\|^p_{W^{1,p}(\Omega_\gamma)}\\
\underbrace{=}_{\text{by the definition of $A$}}\langle A(w_{n+1}), w_{n+1}\rangle
\underbrace{=}_{\text{by choosing } v=w_{n+1} \text{ in } \eqref{its}}\mu_n\langle B (w_n), w_{n+1}\rangle\\
\underbrace{\leq}_{\text{by } (H_4) \text{ of Lemma \ref{auxlmab}}}\mu_n \|w_n\|^{p-1}_{L^p_w(\partial\Omega_\gamma)}\|w_{n+1}\|_{L^p_w(\partial\Omega_\gamma)}
\underbrace{=}_{\text{since }\|w_n\|_{L^p_{w}(\partial\Omega_\gamma)}=1}\mu_n,
\end{multline}
where the last equality above follows due to the fact that $\|w_j\|_{L^p_w(\partial\Omega_\gamma)}=1 \text{ for } j=n,n+1$.
\vskip 0.2cm
\noindent
$(b)$ We observe that
\begin{multline}
\label{its4}
\mu_n \underbrace{=}_{\text{since }\|w_n\|_{L^p_{w}(\partial\Omega_\gamma)}=1}\mu_n \|w_n\|^q_{L^p_{w}(\partial\Omega_\gamma)}
\underbrace{=}_{\text{by the definition of $B$}}\mu_n \langle B(w_n), w_n\rangle\\
\underbrace{=}_{\text{by choosing } v=w_n \text{ in } \eqref{its}}\langle A(w_{n+1}),
w_n\rangle
\underbrace{=}_{\text{by } (H_3) \text{ of Lemma \ref{auxlmab}}} \|w_{n+1}\|^{p-1}_{W^{1,p}(\Omega_\gamma)}\|w_n\|_{W^{1,p}(\Omega_\gamma)}\\
\underbrace{\leq}_{\text{by }\eqref{its3}}\mu_{n}^\frac{p-1}{p}\mu_{n-1}^\frac{1}{p}.
\end{multline}
Therefore, the above inequalities along with \eqref{its3} gives
$$
\|w_{n+1}\|_{W^{1,p}(\Omega_\gamma)}\leq \|w_n\|_{W^{1,p}(\Omega_\gamma)}\text{ and }\mu_n\leq\mu_{n-1}.
$$
Combining the above facts with 
$$
\mu_n\geq \|w_{n+1}\|^p_{W^{1,p}(\Omega_\gamma)}\geq \lambda_p,
$$
which follows by \eqref{its3}, we obtain that the numerical sequences $\{\mu_n\}_{n\in\mathbb{N}}$ and $\{\|w_{n+1}\|^p_{W^{1,p}(\Omega_\gamma)}\}_{n\in\mathbb{N}}$ are convergent. 
Passing to the limit in \eqref{its4} as $n\to\infty$, the sequences  $\{\mu_n\}_{n\in\mathbb{N}}$ and $\{\|w_{n+1}\|^p_{W^{1,p}(\Omega_\gamma)}\}_{n\in\mathbb{N}}$ converges to the same limit, which we denote by $\mu$. Moreover, $\mu\geq\lambda_p$ follows from \eqref{its3}.
\vskip 0.2cm
\noindent
$(c)$ Taking into account Lemma \ref{Xuthm}, Theorem \ref{trace}, Lemma \ref{newlem} along with  Lemma \ref{auxlmab} the result follows proceeding the lines of the proof of \cite[Lemma 5]{Ercole}.

\textbf{Proof of Theorem \ref{subopthm1}:}
The proofs for both part \( (a) \) and \( (b) \) follow exactly as in the proof of \cite[Proposition 2]{Ercole}.
For convenience of the reader, we give the proof below with a brief sketch for part $(b)$.
\vskip 0.2cm
\noindent
$(a)$
By Lemma \ref{Xuthm}, $W^{1,p}(\Omega_\gamma)$ is reflexive and by Theorem \ref{trace}, it is compactly embedded in $L^p_{w}(\partial\Omega_\gamma)$. Therefore, since $\{u_n\}_{n\in\mathbb{N}}$ is bounded in $W^{1,p}(\Omega_\gamma)$, there exists a subsequence $\{u_{n_j}\}_{j\in\mathbb{N}}$ and $u\in W^{1,p}(\Omega_\gamma)\cap L^p_{w}(\partial\Omega_\gamma)$ such that $u_{n_j}\rightharpoonup u$ weakly in $W^{1,p}(\Omega_\gamma)$ and $u_{n_j}\rightarrow u$ strongly in $L^p_{w}(\partial\Omega_\gamma)$. Therefore, by the above strong convergence, we have $\|u\|_{L^p_{w}(\partial\Omega_\gamma)}=\lim_{j\to\infty}\|u_{n_j}\|_{L^p_{w}(\partial\Omega_\gamma)}=1$. Moreover, the above weak convergence gives us
$$
\|u\|_{W^{1,p}(\Omega_\gamma)}\leq \lim_{j\to\infty}\|u_{n_j}\|_{W^{1,p}(\Omega_\gamma)}=\lambda_p^\frac{1}{p}\leq \|u\|_{W^{1,p}(\Omega_\gamma)},
$$
where the last inequality above follows by the definition of $\lambda_p$ from \eqref{la}. Therefore, $\lambda_p=\|u\|^p_{W^{1,p}(\Omega_\gamma)}$. The above inequalities also gives that $\lim_{j\to\infty}\|u_{n_j}\|_{W^{1,p}(\Omega_\gamma)}=\|u\|_{W^{1,p}(\Omega_\gamma)}$. Hence, from Lemma \ref{Xuthm}, by the uniform convexity of $W^{1,p}(\Omega_\gamma)$, we obtain that $u_{n_j}\rightarrow u$ strongly in $W^{1,p}(\Omega_\gamma)$.
\vskip 0.2cm
\noindent
$(b)$ Taking into account Lemma \ref{auxlmab} and proceeding along the lines of the proof of \cite[Proposition 2]{Ercole}, the result follows. 

\vskip 0.4cm

\noindent
{\bf Acknowledgments.} 

\vskip 0.2cm
\noindent
The authors thank Iosif Polterovich for very useful and valuable remarks.

\vskip 0.2cm
\noindent
The authors thank the anonymous reviewer for carefully reading the paper and providing valuable comments.

\vskip 0.2cm
\noindent
The first author is supported by the seed grant: IISERBPR/RD/OO/2024/15, Date:
February 08, 2024, from IISER Berhampur.

\vskip 0.2cm
\noindent
\textbf{Data availability statements.}
\noindent
Data sharing not applicable to this article as no datasets were generated or analysed during the current study.

\vskip 0.3cm

\noindent {\textsf{Prashanta Garain\\
Department of Mathematical Sciences,\\
Indian Institute of Science Education and Research Berhampur\\
Permanent Campus, At/Po:-Laudigam,\\
Dist.-Ganjam, Odisha, India-760003}\\
\textsf{e-mail}: pgarain92@gmail.com\\

\vskip 0.1cm

\noindent {\textsf{Vladimir Gol'dshtein\\
Department of Mathematics,\\
Ben-Gurion University of the Negev,\\ P.O.Box 653, Beer Sheva, 8410501, Israel}\\
\textsf{e-mail}: vladimir@math.bgu.ac.il\\

\vskip 0.1cm

\noindent {\textsf{Alexander Ukhlov\\
Department of Mathematics,\\
Ben-Gurion University of the Negev,\\ P.O.Box 653, Beer Sheva, 8410501, Israel}\\
\textsf{e-mail}: ukhlov@math.bgu.ac.il\\


\begin{thebibliography}{10}

\bibitem{AA} W.~Arendt, A.F.M.~ter Elst, The Dirichlet-to-Neumann operator on rough domains, J. Differential Equations, 251 (2011), 2100--2124.

\bibitem{AM} W.~Arendt, R.~Mazzeo, Friedlander's eigenvalue inequalities and the Dirichlet-to-Neumann semigroup, Commun. Pure Appl. Anal., 11 (2012), 2201--2212. 

\bibitem{AR} K.~Arfi, A.~Rozanova-Pierrat, Dirichlet-to-Neumann or Poincar\'e-Steklov operator on fractals described by $d$-sets, Discrete Continuous Dynamical Systems -- S, 12 (2019), 1--26.

\bibitem{AM20} M.~G.~Armentano, A.~L.~Lombardi, The Steklov eigenvalue problem in a cuspidal domain, Numer. Math., 144 (2020), 237--270.

\bibitem{B21} L.~Barbu, Eigenvalues for anisotropic {$p$}-{L}aplacian under a {S}teklov-like boundary condition, Stud. Univ. Babe\c{s}-Bolyai Math., 66 (2021), 85--94.

%\bibitem{Rossi} J.~F.~Bonder, J.~D.~Rossi, Existence results for the $p$-laplacian with nonlinear boundary condition, J. Math. Anal. appl. %263 (2001), no. 1, 195-223.

\bibitem{BS} S.~Bergmann, M.~Schiffer, Kernel function and elliptic differential equations in mathematical physics, New York: Academic, (1953).

\bibitem{Br10} H.~Brezis, Functional analysis, Sobolev spaces and partial differential equations. Springer, New York, 2010. xiv+599 pp.

\bibitem{BR} J.~F.~Bonder, J.~D.Rossi, Existence results for the $p$-Laplacian with nonlinear boundary conditions, Journ. Math. Anal. Appl., 263 (2001), 195--223.

\bibitem{IC} I.~Cioranescu, Geometry of Banach Spaces, Duality Mappings and Nonlinear Problems, Springer, Dordrecht, 1990. xiv+260 pp.

\bibitem{CGGS} B.~Colbois, A.~Girouard, C.~Gordon, D.~Sher, Some recent developments on the Steklov eigenvalue problem, Revista Mat. Compl., 37 (2024), 1--161.

\bibitem{CPV} C.~Conca, J.~Planchard, M.~Vanninathan, Fluid and periodic structures, New York: Wiley, (1995).

\bibitem{var} P.~G.~Ciarlet, Linear and nonlinear functional analysis with applications, Society for Industrial and Applied Mathematics, Philadelphia, PA, (2013).

\bibitem{Dama} L.~Damascelli, Comparison theorems for quasilinear degenerate elliptic operators and applications to symmetry and monotonicity results, Ann. Inst. h. Poincar\'e Anal. Non Lin\'eaire, 15 (4) (1998), 493-516.

\bibitem{Dav} E.~B.~Davies, Spectral Theory and Differential Operators, Cambridge University Press, Cambridge, (1995).

\bibitem{Ercole} G.~Ercole, Solving an abstract nonlinear eigenvalue problem by the inverse iteration method, Bull. Braz. Math. Soc. (N.S.), 49 (2018), 577--591.

\bibitem{Fe69} H.~Federer, Geometric measure theory, Sp\-rin\-ger Verlag, Berlin, (1969).

\bibitem{FL} A.~Ferrero, P.~D.~Lamberti, Spectral stability of the Steklov problem, Nonlinear Anal., 222 (2022), 112989.

\bibitem{GP17} A.~Girouard, I.~Polterovich, Spectral geometry of the Steklov problem, J. Spectr. Theory, 7 (2017), 321--359.

\bibitem{GG94} V.~Gol'dshtein, L.~Gurov, Applications of change of variables operators for exact embedding theorems, Integral Equations Operator Theory 19 (1994), 1--24.

\bibitem{GU09} V.~Gol'dshtein, A.~Ukhlov, Weighted Sobolev spaces and embedding theorems, Trans. Amer. Math. Soc., 361 (2009), 3829--3850.

\bibitem{GU17} V.~Gol'dshtein, A.~Ukhlov, The spectral estimates for the Neumann-Laplace operator in space domains. 
Adv. in Math., 315 (2017), 166--193.

\bibitem{GV} V.~Gol'dshtein, M.~Ju.~Vasiltchik, Embedding theorems and boundary-value problems for cusp domains, Trans. Amer. Math. Soc., 362 (2010), 1963--1979.

\bibitem{KLP} M.~Karpukhin, J.~Lagac\'e, I.~Polterovich, Weyl's law for the Steklov problem on surfaces with rough boundary, Arch. Rational Mech. Anal., (2023) 247:77.

\bibitem{KUZ} P.~Koskela, A.~Ukhlov, Zh.~Zhu, The volume of the boundary of a Sobolev $(p,q)$-extension domain, J. Funct. Anal., 283 (2022), 109703.

\bibitem{MN10} N.~Mavinga, M.~N.~Nkashama, Steklov--Neumann eigenproblems and nonlinear elliptic equations with nonlinear boundary conditions, J. Differential Equations, 248 (2010), 1212--1229.

\bibitem{M} V.~Maz'ya, Sobolev spaces: with applications to elliptic partial differential equations, Springer: Berlin/Heidelberg, (2010).

\bibitem{MP} V.~G.~Maz'ya, S.~V.~Poborchi, Differentiable functions on bad domains, World Scientific Publishing Co., River Edge, NJ, (1997).

\bibitem{NT} S.~A.~Nazarov, J.~Taskinen, On the spectrum of the Steklov problem in a domain with a peak, Vestnik St. Petersburg Univ. Math., 41 (2008), 45--52.

\bibitem{RSS} G.~V.~Rozenblum, M.~A.~Shubin, M.~Z.~Solomyak, Spectral Theory of Differential Operators, Springer: Berlin/Heidelberg, (1994).

\bibitem{Steklov} W.~Stekloff, Sur les probl\'emes fondamentaux de la physique mathematique, Annales Sci. ENS, S\'er. 3, 19 (1902), 191--259 and 455--490.

\bibitem{XW} Ch.~Xia, Q.~Wang, Inequalities for the Steklov eigenvalues, Chaos, Solitons Fractals, 48 (2013), 61--67.


\end{thebibliography}
\end{document}